\documentclass[11pt]{amsart}
\usepackage{amsmath} 
\usepackage{amsthm}
\usepackage{amssymb} 
\usepackage{amsfonts}
\usepackage{graphicx}
\usepackage{xcolor}
\definecolor{darkgreen}{rgb}{0,0.5,0}
\usepackage[normalem]{ulem}
\usepackage{comment}

\usepackage{caption,subcaption,pinlabel}
\usepackage{graphics}

\newtheorem{theorem}{Theorem}[section]
\newtheorem{lemma}[theorem]{Lemma}

\newtheorem{question}[theorem]{Question}

\newtheorem{corollary}[theorem]{Corollary}

\theoremstyle{definition}
\newtheorem{remark}[theorem]{Remark}
\newtheorem{definition}[theorem]{Definition}

\newcommand{\N}{\mathbb{N}}
\newcommand{\Z}{\mathbb{Z}}
\newcommand{\Q}{\mathbb{Q}}

\let\int\relax
\newcommand{\int}{\mathring}

\newcommand{\boundary}{\partial}

\usepackage{accents}

\DeclareMathSymbol{\wtilde}{\mathord}{largesymbols}{"65}

\title[0-concordance monoid has an infinite linearly independent set]{The 0-concordance monoid admits an infinite linearly independent set}
\author{Irving Dai}
\address{Department of Mathematics\\Stanford University\\Stanford, CA 94301, USA}
\email{irvingfdai@gmail.com}
\author{Maggie Miller}
\address{Department of Mathematics\\Stanford University\\Stanford, CA 94301, USA}
\email{maggie.miller.math@gmail.com}

\thanks{At the time of research, ID was supported by NSF grant DGE-1148900 and MM was supported by NSF grant DGE-1656466, both at Princeton University.}


\begin{document}

\maketitle

\begin{abstract}
Under the relation of $0$-concordance, the set of knotted 2-spheres in $S^4$ forms a commutative monoid $\mathcal{M}_0$ with the operation of connected sum. Sunukjian has recently shown that $\mathcal{M}_0$ contains a submonoid isomorphic to $\Z^{\ge 0}$. In this note, we show that $\mathcal{M}_0$ contains a submonoid isomorphic to $(\Z^{\ge 0})^\infty$. Our argument relates the $0$-concordance monoid to linear independence of certain Seifert solids in the spin rational homology cobordism group. 
\end{abstract}

\section{Introduction}\label{sec:intro}
In this note, we study a restricted notion of concordance between 2-knots in $S^4$, called $0$-concordance. Introduced by Melvin in~\cite{melvin}, the relation of $0$-concordance turns the set of knotted 2-spheres in $S^4$ into a monoid under the operation of connected sum. We call this the \textit{$0$-concordance monoid} and denote it by $\mathcal{M}_0$. In~\cite{melvin}, Melvin showed that Gluck twists along 0-concordant 2-knots produce diffeomorphic homotopy 4-spheres. It is thus a natural question to ask about the size of $\mathcal{M}_0$. See Section~\ref{sec:background} for definitions and further discussion.

In~\cite{sunukjian}, Sunukjian showed that $\mathcal{M}_0$ is nontrivial, and gave an explicit submonoid isomorphic to $\Z^{\ge 0}$. In this note, we show that a straightforward extension of his argument gives the following method for studying $\mathcal{M}_0$:

\begin{theorem}\label{mainthm}
Let $S_0$ and $S_1$ be 2-knots, and suppose that $S_0$ and $S_1$ bound embedded punctured rational homology spheres $\smash{\int{Y_0}}$ and $\smash{\int{Y_1}}$, respectively. If $S_0$ and $S_1$ are $0$-concordant, then there exists a spin rational homology cobordism from $Y_0$ to $Y_1$, with the spin structure on each $Y_i$ induced by restricting the trivial spin structure on $S^4$ to $\mathring{Y}_i$.
\end{theorem}
\noindent
The proof of Theorem~\ref{mainthm} is an easy consequence of previous techniques used by various authors (see for example \cite{sunukjian, ruberman2}), but we have been unable to find a statement of this result in the literature.

Theorem~\ref{mainthm} immediately allows us to obtain many linearly independent families in $\mathcal{M}_0$.\footnote{Here, by a \textit{linear relation in} $\mathcal{M}_0$, we mean a 0-concordance between two connected sums $\# c_i S_i$ and $\#c_j S_j$, with all coefficients non-negative (since $\mathcal{M}_0$ is a monoid). Linear independence is then defined in the obvious way.} Indeed, let $\{S_i\}$ be a family of 2-knots, and suppose that each $S_i$ bounds an embedded Seifert solid $\smash{\int Y_i}$, where $Y_i$ is a rational homology sphere. Then any (non-negative) connected sum $\# c_i S_i$ bounds the punctured rational homology sphere $\smash{\int Y}$, where $Y = \# c_i Y_i$. It follows from Theorem~\ref{mainthm} that a nontrivial linear relation among the $[S_i]$ in $\mathcal{M}_0$ induces a nontrivial linear relation among the $[Y_i]$ in the spin rational homology cobordism group $\smash{\Theta^{spin}_{\Q}}$.

It thus suffices to find families of 2-knots admitting Seifert solids which are linearly independent in $\smash{\Theta^{spin}_{\Q}}$. We list several such families which follow more-or-less directly from various results appearing in the literature. Our first example is an immediate consequence of work of Aceto, Celoria, and Park~\cite{daniele}, who studied rational homology cobordisms between lens spaces:

\begin{corollary}\label{2bridgecor}
Let $\mathcal{T}$ be any linearly independent family of 2-bridge knots in the (classical) knot concordance group. Let
\[
\mathcal{F} = \{[S] \mid S \text{ is the 2-twist spin of } K \in \mathcal{T}\}.
\]
Then $\mathcal{F}$ is linearly independent in $\mathcal{M}_0$. For example, we may take $\mathcal{T}$ to be the set of torus knots $T(2, p)$, for $p \geq 3$ odd.
\end{corollary}
\begin{proof}
By work of Zeeman~\cite{zeeman}, the 2-twist spin of a knot $K$ bounds (a punctured copy of) its double branched cover $\Sigma_2(K)$. According to~\cite[Proposition 2.7]{daniele}, there is an isomorphism
\[
\beta_2|_\mathcal{B} : \mathcal{B} \rightarrow \left<\{L(r,s) \mid r \text{ odd}\}\right>
\]
where on the left, $\mathcal{B}$ is the subgroup of the classical knot concordance group generated by 2-bridge knots, and on the right we have the subgroup of the rational homology cobordism group $\Theta_{\Q}$ generated by lens spaces $L(r, s)$ for $r$ odd. The map $\beta_2$ is given by taking double branched covers. In particular, taking the double branched covers of any linearly independent family $\mathcal{T} \subseteq \mathcal{B}$ yields a linearly independent family of lens spaces in $\Theta_{\Q}$. These lens spaces are still linearly independent under the more restrictive relation of spin rational homology cobordism, so applying Theorem~\ref{mainthm} gives the desired result. The last part of the claim follows from the well-known fact that the 2-bridge torus knots $T(2, p)$ ($p>1$ odd) are linearly independent in the knot concordance group.
\end{proof}

In Corollary~\ref{2bridgecor}, we did not use the fact that Theorem~\ref{mainthm} produces a \textit{spin} rational homology cobordism. Spin rational homology cobordisms can be studied via Floer-theoretic techniques such as involutive Heegaard Floer homology; see \cite{hendricksmanolescu, hendricksmanolescuzemke}. Using this, in~\cite{irvingmatt} a systematic method is presented for obstructing spin rational homology cobordisms between linear combinations of Seifert fibered integer homology spheres. This is especially useful since the Brieskorn spheres $\Sigma(p, q, r)$ (with $p$, $q$, and $r$ positive and coprime) arise naturally as Seifert solids of the following twist-spun 2-knots:
\begin{enumerate}
\item $\Sigma(p, q, r)$ is the $p$-fold branched cover of the torus knot $T(q, r)$, and hence (by~\cite{zeeman}) is a Seifert solid for the $p$-twist spin of $T(q, r)$
\item $\Sigma(p, q, r)$ is the double branched cover of the pretzel knot $P(p, q, r)$, and hence (by~\cite{zeeman})  is a Seifert solid for the 2-twist spin of $P(p, q, r)$
\end{enumerate}
Note that in general, the $p$-twist spin of $T(q, r)$ is not the same as the 2-twist spin of $P(p, q, r)$; see~\cite{gordon, plotnick2}.

\begin{remark}
As written, many results derived from involutive Heegaard Floer homology (for example, the results of~\cite{irvingmatt}) deal with the integer homology cobordism group. However, these methods generally carry over without change to study spin rational homology cobordism. Briefly, the approach in question utilizes the involutive Heegaard Floer package of Hendricks and Manolescu. This assigns an invariant to any 3-manifold $Y$ equipped with a self-conjugate spin$^c$-structure; we similarly obtain induced cobordism maps whenever $W$ is a cobordism equipped with self-conjugate spin$^c$-structure. If $Y$ and $W$ have only one self-conjugate spin$^c$-structure, then one does not have to specify which spin$^c$-structure on $Y$ one is working in (or compute what a given self-conjugate spin$^c$-structure on $W$ restricts to at the ends). It is thus more convenient to work with integer or $\Z/2\Z$-homology cobordism, as in \cite[Theorem 1.3]{hendricksmanolescu} or \cite[Proposition 1.5]{hendricksmanolescuzemke}. However, such methods apply more generally; here, we will allow spin rational homology cobordism but only consider examples in which the ends are integer homology spheres.
\end{remark}

We thus immediately obtain:

\begin{corollary}\label{cor}
Let $\{\Sigma(p_i, q_i, r_i)\}_{i \in \N}$ be any family of Brieskorn integer homology spheres whose elements are linearly independent in $\Theta^{spin}_{\Q}$.\footnote{Note that we require $\Sigma(p_i, q_i, r_i)$ to be an integer homology sphere, even though we consider spin rational homology cobordism. This means that for each $i$, we have that $p_i$, $q_i$, and $r_i$ are coprime.} Then the families
\[
\mathcal{F}_1 = \{p_i\text{-twist spin of }T(q_i, r_i)\}_{i \in \N}
\]
and
\[
\mathcal{F}_2 = \{2\text{-twist spin of } P(p_i, q_i, r_i)\}_{i \in \N}
\]
are each linearly independent in $\mathcal{M}_0$. For example, we may choose 
\begin{align*}
p_i &= 2i+1, \\
q_i &= 4i+1, \text{ and} \\
r_i &= 4i+3.
\end{align*}
\end{corollary}
\begin{proof}
It is known that the Brieskorn spheres $\Sigma(2i+1, 4i+1, 4i+3)$ (for $i \geq 1$) are linearly independent in $\Theta^{spin}_{\Q}$. (See for example \cite{irvingmatt}; as written, this obstructs integer homology cobordism, but see the previous remark.)
\end{proof}

Another interesting linearly independent family is given by considering the pretzel knots with
\begin{align*}
p_k &= -2k-1, \\
q_k &= 3k+2, \text{ and} \\
r_k &= 6k+1,
\end{align*}
for $k \geq 1$ odd. According to \cite[Lemma 5.1]{cochran}, the double branched cover of $P(p_k, q_k, r_k)$ is $\Sigma(|p_k|, q_k, r_k)$. In \cite{savk}, it is shown that the Brieskorn spheres $\Sigma(|p_k|, q_k, r_k)$ are linearly independent in the homology cobordism group. (Again, \cite{savk} deals with integer homology cobordism, but the methods of that paper carry over without change to spin rational homology cobordism.) Thus the 2-twist spins of the $P(p_k, q_k, r_k)$ are linearly independent in $\mathcal{M}_0$.

Taking any of the families discussed above, we have:
\begin{theorem}\label{0thm}
The 0-concordance monoid contains an infinitely generated submonoid isomorphic to $(\Z^{\ge 0})^\infty$.
\end{theorem}

\begin{remark}\label{jasonremark}
In a recent preprint, Jason Joseph \cite{jason} has shown that $\mathcal{M}_0$ is not a group -- that is, there exist non-invertible elements of $\mathcal{M}_0$. His approach is to construct a homomorphism $\Delta$ from $\mathcal{M}_0$ to $\mathcal{I}(\mathbb{Z}[t^{\pm1}])$, the ideal class monoid of $\mathbb{Z}[t^{\pm1}]$, in which no nontrivial element is invertible. This homomorphism is induced by Alexander ideal. In particular, $\Delta([K])$ is trivial when $K$ has trivial Alexander ideal.

Any 2-knot $K$ as in Corollary \ref{cor} has trivial Alexander ideal, since the complement of $K$ is fibered by integer homology spheres, implying that the commutator subgroup of $\pi_1(S^4\setminus K)$ is perfect. Therefore, Corollary \ref{cor} implies that there is a $(\mathbb{Z}^{\ge 0})^\infty$ submonoid of Ker$(\Delta)$.

Conversely, Joseph provides examples of 2-knots which he can obstruct from being 0-concordant to the unknot but for which the approach used in this paper fails, e.g. the 2-twist spun Stevedore knot.
\end{remark}

\begin{question}
Is there torsion in $\mathcal{M}_0$?
\end{question}

\begin{remark}
Aceto, Celoria, and Park \cite{daniele} provide many examples of 2-torsion in the spin rational homology cobordism group. Their examples come from double branched covers of negative amphichiral 2-bridge knots (e.g.\ the double branched cover $L(5,2)$ of the figure eight knot). One might ask whether the classes of corresponding 2-twist spun knots (e.g.\ the 2-twist spun figure eight) are 2-torsion in ${\mathcal{M}}_0$. However, Joseph \cite{jason} proves that the class of a 2-twist spun knot of a 2-bridge knot is not even invertible in ${\mathcal{M}_0}$, so certainly cannot be torsion.
\end{remark}

\subsection*{Acknowledgements}
Thanks to Jason Joseph for interesting conversations about 0-concordance and to Nathan Sunukjian for helpful comments and observations. Thanks also to an anonymous referee for several helpful comments, including suggesting Remark \ref{jasonremark}.

\section{Background and Definitions}\label{sec:background}
In this section, we briefly review some relevant definitions and give some motivation for the study of 0-concordance. Recall that an oriented 2-sphere which is smoothly embedded in $S^4$ is called a \emph{2-knot}. We say that two 2-knots $K_0$ and $K_1$ are \emph{concordant} if there is a smoothly embedded cylinder $S^2 \times I \subseteq S^4 \times I$ whose boundary is the disjoint union of $K_0 \times \{0\}$ and $K_1 \times \{1\}$. In this note, we study a restricted notion of concordance introduced by Melvin~\cite{melvin}. One reason for this is the following well-known theorem of Kervaire~\cite{kervaire}:

\begin{theorem}[Kervaire~\cite{kervaire}]
Every 2-sphere smoothly embedded in $S^4$ is the boundary of a 3-ball smoothly embedded in $B^5$.
\end{theorem}
\noindent
Thus, even though knot concordance is a rich subject in the classical case of knots in $S^3$, the concordance group of knotted 2-spheres in $S^4$ is trivial. We thus instead have:

\begin{definition}[Melvin~\cite{melvin}]\label{0def}
We say that two 2-knots $K_0$ and $K_1$ in $S^4$ are {\emph{0-concordant}} if there exists a smooth embedding $f:S^2\times I\to S^4\times I$ such that:
\begin{enumerate}
    \item We have $f(S^2\times \{0\})=K_0\times \{0\}$ and $f(S^2\times \{1\})=K_1\times \{1\}$; that is, $f$ constitutes a concordance.
    \item The height function $h: S^4 \times I \rightarrow I$ (given by projection onto the second factor) is Morse when restricted to the image $f(S^2\times I)$.
    \item For each $t\in[0,1]$, the cross-section $f(S^2\times I)\cap h^{-1}(t)$ is either singular or a disjoint union of $2$-spheres.
\end{enumerate}
\noindent
If $K$ is $0$-concordant to the unknotted $2$-sphere, then we say $K$ is {\emph{0-slice}}.
\end{definition}

Using Definition~\ref{0def}, we form a commutative monoid $\mathcal{M}_0$ whose elements are equivalence classes of
$2$-knots under the relation of $0$-concordance. The monoid operation is given by connected sum, while the identity is given by the class of the unknot in $S^4$. Note that if $K$ is a 2-knot, then we do \textit{not} necessarily expect $K\#-\overline{K}$ to be $0$-slice. 

In~\cite{melvin}, Melvin showed that if $K_0$ and $K_1$ are $0$-concordant, then the Gluck twists of $S^4$ about $K_0$ and $K_1$ are diffeomorphic homotopy 4-spheres. This leads to the following natural question:

\begin{question}[Kirby Problem 1.105(A)~\cite{kirby} (partial)]\label{kirbyprob}
Is every 2-knot 0-slice?
\end{question}

\noindent
In~\cite{sunukjian}, Sunukjian gave a negative answer to Question~\ref{kirbyprob} by producing an infinite family of 2-knots which are distinct up to 0-concordance. Sunukjian's argument utilized the following obstruction:

\begin{theorem}[Sunukjian~\cite{sunukjian}]\label{nathanthm}
Let $S_0$ and $S_1$ be 2-knots, and suppose that $S_0$ and $S_1$ bound embedded punctured rational homology spheres $\smash{\int{Y_0}}$ and $\smash{\int{Y_1}}$, respectively. If $S_0$ and $S_1$ are $0$-concordant, then $d(Y_0,\mathfrak{s_0})=d(Y_1,\mathfrak{s_1})$, where $\mathfrak{s_i}$ is the spin$^c$-structure on $Y_i$ induced from the trivial spin$^c$-structure on $S^4.$
\end{theorem}

\noindent
Note that since the $d$-invariant is a rational homology cobordism invariant, Theorem~\ref{mainthm} is evidently a strengthening of Theorem~\ref{nathanthm}.
\begin{corollary}[Sunukjian~\cite{sunukjian}]
The 0-concordance monoid contains a submonoid isomorphic to $\Z^{\ge 0}$.
\end{corollary}
\begin{proof}
This follows immediately as in Section~\ref{sec:intro} by producing a Seifert solid with nonzero $d$-invariant. Indeed, let $S$ be the $5$-twist spun trefoil. Then $S$ bounds a punctured Poincar\'{e} homology sphere $\smash{\int{Y}}$ in $S^4$. (As described previously, the $5$-fold cover of $S^3$ branched along $T(2, 3)$ is the Poincar\'{e} homology sphere $\Sigma(2, 3, 5)$.) Since $\#n S$ bounds a punctured copy of $\#nY$ and $d(nY) = 2n$, it follows from Theorem~\ref{nathanthm} that $[S]$ is not torsion in $\mathcal{M}_0$.
\end{proof}

\section{Proof of Theorem~\ref{mainthm}}

We now turn to the proof of Theorem~\ref{mainthm}. For convenience of the reader, we describe the geometric setup of Sunukjian~\cite{sunukjian}:

\begin{definition}
Given 2-knots $S_0$ and $S_1$ in $S^4$, we say that there is a {\emph{ribbon concordance from $S_0$ to $S_1$}} if 
there exists a smooth embedding $f:S^2\times I\to S^4\times I$ such that:
\begin{enumerate}
    \item We have $f(S^2\times \{0\})=S_0\times \{0\}$ and $f(S^2\times \{1\})=S_1\times \{1\}$; that is, $f$ constitutes a concordance.
    \item The height function $h: S^4 \times I \rightarrow I$ (given by projection onto the second factor) is Morse when restricted to the image $f(S^2\times I)$.
    \item Every critical point of $h$ restricted to $f(S^2 \times I)$ is either of index-0 or index-1.
\end{enumerate}
\noindent
If there is a ribbon concordance from the unknotted $2$-sphere to $K$, then we say $K$ is {\emph{ribbon}}.
\end{definition}

Note that unlike 0-concordance, ribbon concordance is not a symmetric relation. Sunukjian's observation was that a 0-concordance can be decomposed into two ribbon concordances, each from one end of the original 0-concordance to the 2-knot in the ``middle cross-section". We recall his argument here:

\begin{lemma}[Sunkujian \cite{sunukjian2}]\label{commonribbon}
If $S_0$ and $S_1$ are 0-concordant 2-knots, then there exists a 2-knot $R$ so that there is a ribbon concordance from $K_0$ to $R$ and also a ribbon concordance from $K_1$ to $R$.
\end{lemma}

\begin{proof}
Let $M$ denote the embedded cylinder $f(S^2 \times I)$ afforded in the definition of 0-concordance. The Morse function $h|_M$ induces a handle decomposition of $M$ relative to $S_0\times \{0\}$. Because $M$ is connected, there are at least as many 1-handles as 0-handles in this decomposition.

If there are more 1-handles than 0-handles, then (since $M$ has no first homology) some 1-handle $H_1$ must cancel algebraically with a higher 2-handle $H_2$ corresponding to an index-2 critical point of $f$ whose image lies in some slice $S^4\times \{t_0\}$. Then the intersection $\Sigma := M \cap (S^4\times\{t_0-\epsilon\})$ contains the belt sphere of $H_1$ and the attaching sphere of $H_2$. These circles intersect algebraically once, so must be essential in $\Sigma$. This contradicts the fact that $\Sigma$ is a genus zero surface. By contradiction, there must be an equal number of 0- and 1-handles in $M$. That is, $h|_M$ has an equal number of index-0 and index-1 points.

Isotope $f$ to reorder the critical points of $h|_M$ to be in order of increasing index, with the index-0 and index-1 critical points in $S^4\times(0,1/4)$ and the index-2 and index-3 critical points in $S^4\times(3/4,1)$. Let $R = M\cap(S^4\times \{1/2\})$. Then $M\cap (S^4\times[0,1/2])$ is a ribbon concordance from $K_0$ to $R$, and $(S^4\times[1/2,1])$ (when viewed backwards) is a ribbon concordance from $K_1$ to $R$.
\end{proof}

The crucial point of Lemma~\ref{commonribbon} is that Seifert solids for $S_0$ and $S_1$ induce Seifert solids for $R$, as follows:

\begin{lemma}[Sunukjian \cite{sunukjian}, Yanagawa \cite{yanagawa}]\label{commonmfd}
Let $S$ be a 2-knot which bounds a punctured Seifert solid $\smash{\int{Y}}$ in $S^4$. Suppose that there is a ribbon concordance from $S$ to $R$. Then $R$ bounds $\smash{\int{Y}} \#_k(S^1\times S^2)$ for some non-negative integer $k$.
\end{lemma}
\begin{proof}
Using the ribbon concordance from $S$ to $R$, we see that $R$ bounds an immersed copy of $\smash{\int{Y}}$, which we denote by $V$. More precisely, $V$ is obtained from $\smash{\int{Y}}$ by adding disjoint 3-balls (one for each index-0 point of the ribbon concordance) and then tubing these to $\smash{\int{Y}}$ along three-dimensional 1-handles (one for each index-1 point of the ribbon concordance). Note that the 1-handles may intersect the interior of $\smash{\int{Y}}$ and/or the interiors of the 3-balls. All intersections are disks, as in the top of Figure \ref{fig:cutandpaste}.

We then resolve the self-intersections of $V$ by cut-and-paste surgery to find a 3-manifold $\int{M}$ embedded in $S^4$ with $\boundary\int{M} = R$, as in Figure~\ref{fig:cutandpaste}. (See also the proof of \cite[Theorem 5.2]{ruberman3}.) Each cut-and-paste operation replaces a 3-ball in $V$ (containing a region of self-intersection in its interior) with a punctured $S^1\times S^2$. Hence we see that $\int{M} \cong \smash{\int{Y}}\#_{k}(S^1\times S^2)$ for some $k\ge 0.$ 
\end{proof}

\begin{figure}
    \centering
    \labellist
\small\hair 2pt
\pinlabel {time} at 350 -10
\endlabellist
    \includegraphics[width=80mm]{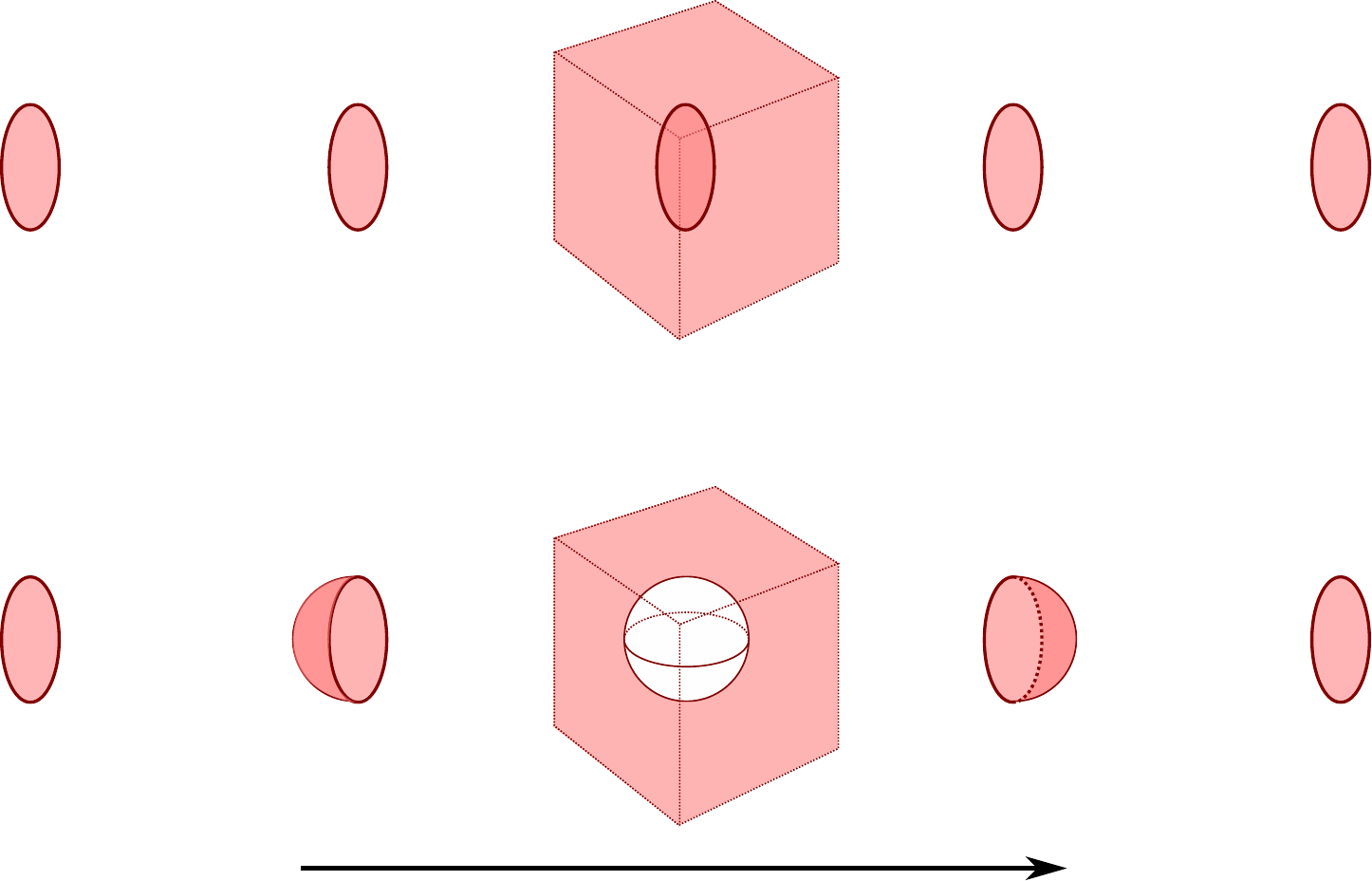}\\
    \vspace{.2in}
    \includegraphics[width=70mm]{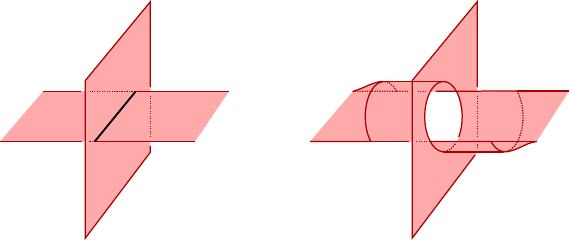}
    \caption{{\bf{Top:}} a movie of a ribbon self-intersection of a 3-manifold $V$ ribbon-immersed in $S^4$. The fourth dimension is taken to be time; at each time we see a 3-dimensional picture. In this picture there are two local sheets of $V$ which intersect in a disk. {\bf{Middle:}} We resolve the ribbon self-intersections of $V$ to find an embedded submanifold $\smash{\int{M}}$ of $S^4$. We have $\smash{\int{M}}\cong \smash{\int{Y}}\#_k(S^1\times S^2)$, where $k$ is the number of ribbon self-intersections of $V$ in $S^4$ and $\smash{\int{Y}}$ is the domain of the immersed manifold $V$. {\bf{Bottom:}} We draw the analogous operation a dimension down. Here, we see a ribbon self-intersection of an immersed surface. We may cut-and-paste to remove this intersection at the cost of increasing the surface genus.}
    \label{fig:cutandpaste}
\end{figure}

Now let $S_0$ and $S_1$ be 0-concordant 2-knots that bound Seifert solids $\smash{\int{Y_0}}$ and $\smash{\int{Y_1}}$ in $S^4$. Using Lemmas~\ref{commonribbon} and~\ref{commonmfd}, we obtain a 2-knot $R$ which bounds two Seifert solids given by stabilizations $\smash{\int{Y_0}} \#_{k_0} (S^1 \times S^2)$ and $\smash{\int{Y_1}} \#_{k_1} (S^1 \times S^2)$ for some $k_0, k_1 \ge 0$. For convenience, we denote 
\begin{align*}
M_1 &= Y_0 \#_{k_0} (S^1 \times S^2), \\
M_2 &= Y_1 \#_{k_1} (S^1 \times S^2). 
\end{align*}
Note that $\smash{\int{M_1}}$ and $\smash{\int{M_2}}$ are each individually embedded in $S^4$, although they may intersect each other. We now surger $S^4$ along $R$ to obtain a 4-manifold $X$ with the integer homology of $S^1 \times S^3$. This allows us to cap off $\smash{\int{M_1}}$ and $\smash{\int{M_2}}$ to obtain embeddings of $M_1$ and $M_2$ into $X$. It is easily checked that these are \textit{cross-sections} of $X$; that is, $M_1$ and $M_2$ represent generators of $H_3(X)$.

Developing invariants for such cross-sections has been studied extensively by several authors. Behrens-Golla~\cite{behrens} and Ruberman-Levine~\cite{ruberman} have both defined ``twisted" versions of the Heegaard Floer $d$-invariant which can be used to obstruct whether two given cross-sections can appear at the same time. Indeed, Sunukjian's proof of Theorem~\ref{nathanthm} used the twisted $d$-invariant developed by Behrens-Golla. Here, however, we will sidestep the need for these newer invariants by reducing the question to one of spin rational homology cobordism.

\begin{proof}[Proof of Theorem~\ref{mainthm}]
Let $X$, $M_1$, and $M_2$ be as above. Let $X_p$ be a $p$-fold cover of $X$, where $p$ is some large prime. Then $X_p$ has the rational homology of $S^1\times S^3$ by a result of Sumners~\cite[Theorem 3]{sumners}; if $p$ is large enough, we can moreover find lifts of $M_1$ and $M_2$ in $X_p$ which are disjoint cross-sections. Let $W$ be one component of $X_p \setminus(M_1 \cup M_2)$, so that $W$ is a cobordism from $M_1$ to $M_2$. We claim that $i_*:H_2(\boundary W; \Q) \rightarrow H_2(W; \Q)$ is surjective. This follows immediately from the Mayer-Vietoris sequence
\[
\cdots \rightarrow H_2(\boundary W; \Q) \rightarrow H_2(W; \Q) \oplus H_2(W^c; \Q) \rightarrow H_2(X_p; \Q) = 0, 
\]
where $W^c$ is the other component of $X_p \setminus(M_1 \cup M_2)$. (One can also use the fact that the absolute intersection form on $W$ vanishes identically.) Note that $X$ is spin, since $X$ is a homology $S^1 \times S^3$. Hence $X_p$ is spin, being the cover of a spin manifold; it follows that $W$ is spin, being a codimension-zero submanifold of a spin manifold.

Now attach four-dimensional 3-handles to each end of $W$ along the generators of $H_2(M_1; \Z)$ and $H_2(M_2; \Z)$. That is, recalling that $M_1$ and $M_2$ are stabilizations of the rational homology spheres $Y_0$ and $Y_1$, attach 3-handles along an essential $S^2$ in each $S^1\times S^2$ summand of $M_1$ and $M_2$ to obtain a cobordism $W'$ from $Y_0$ to $Y_1$. Since $i_*:H_2(\boundary W; \Q) \rightarrow H_2(W; \Q)$ is surjective, this has the effect of killing the second homology and making $H_2(W'; \Q) = 0$. It is not difficult to check that $W'$ is also spin. Indeed, each 3-handle attaching region is just $S^2 \times I$. This has a unique spin structure, which evidently extends over the 3-handle.

Without loss of generality, we may further assume that $b_1(W') = 0$ by surgering out representatives of the generators of $H_1(W';\Z)$. More precisely, let $\gamma$ be a curve representing a generator of $H_1(W'; \Z)$, and cut out a neighborhood $\nu(\gamma)$ of $\gamma$ in $W'$. There are two choices of framing when gluing in $D^2\times S^2$ to $\boundary\nu(\gamma) \cong S^1 \times S^2$; we choose the framing so that the spin structure of $W'$ restricted to $\boundary\nu(\gamma)$ extends over the glued-in $D^2\times S^2$. Since $\gamma$ is not rationally nullhomologous, it is easily checked that $H_2(W'';\Q)=0$.

We thus have obtained a spin cobordism $W''$ between $Y_0$ and $Y_1$ with $H_1(W'';\Q) = H_2(W'';\Q) = 0$. Since $\partial W''$ is the disjoint union of two rational homology spheres, it follows from Poincar\'{e} duality that $H_3(W'';\Q)=\Q$. Hence $W''$ is a spin rational homology cobordism from $Y_0$ to $Y_1$.
\end{proof}

\begin{remark}
The proof of Theorem \ref{mainthm} holds in any integer homology 4-sphere $X$. That is, suppose that $S_0$ and $S_1$ are 2-spheres smoothly embedded in $X$ and that $S_0$ and $S_1$ bound bound embedded punctured rational homology spheres $\smash{\int{Y_0}}$ and $\smash{\int{Y_1}}$, respectively. If $S_0$ and $S_1$ are $0$-concordant (i.e.\ there is a concordance in $X\times I$ between $S_0$ and $S_1$ whose regular cross-sections are genus-zero), then there exists a spin rational homology cobordism from $Y_0$ to $Y_1$, with the spin structures on each $Y_i$ induced by the trivial spin structure on $X$ restricting to $\mathring{Y}_i$. Note that any punctured 3-manifold that embeds in $S^4$ also embeds in $X$. We conclude, using the same examples as in Theorem \ref{0thm}, that the monoid of 0-concordance classes of 2-spheres in $X$ contains a $(\mathbb{Z}^{\ge0})^\infty$ submonoid (even if we restrict ourselves to only classes containing a nullhomotopic 2-sphere in $X$).

\end{remark}

\bibliographystyle{abbrv}
\bibliography{biblioshort}

\end{document}